\documentclass[12pt,reqno]{amsart}

\usepackage{amsfonts}
\usepackage{mathrsfs}
\usepackage{amsthm}
\usepackage{amssymb}
\usepackage{amsmath,amsthm,amsfonts}
\usepackage{mathabx}
\usepackage{textcomp}
\usepackage{amsmath,amscd}
\usepackage[all,cmtip]{xy}
\usepackage{mathtools}
\usepackage[margin=1in]{geometry}
\usepackage{lipsum}
\usepackage{extarrows}

\newtheorem{theorem}{Theorem}[section]

\newtheorem{defn}{Definition}[section]
\newtheorem*{theorem*}{Theorem}
\newtheorem{cor}{Corollary}[section]
\theoremstyle{definition}

\newcommand{\me}{\mathcal{M}}        %

\DeclareMathAlphabet{\mathpzc}{OT1}{pzc}{m}{it}
\newcommand{\ri}{\mathrm{Ric}}
\newcommand{\R}{\mathbb{R}}           
\newcommand{\mc}{MC^k}

\begin{document}
\title[A simple proof of the short-time existence and uniqueness for Ricci flow]
      {A simple proof of the short-time existence and uniqueness  for Ricci flow}

\author{Kaveh Eftekharinasab}
\address{
Topology dept. \\ Institute of Mathematics of NAS of Ukraine \\ Te\-re\-shchen\-kivska st. 3, Kyiv, 01601 Ukraine
}
\email{kaveh@imath.kiev.ua}
\keywords{Bounded Fr\'{e}chet manifold, Ricci flow, bounded geometry}
\subjclass[2010]{
53C44, 
58A99 
 }
\begin{abstract}
 In this short note, we give simple proof of the Ricci flow's local existence and uniqueness on closed Einstein manifolds. We suggest a new setting for studying the space of Riemannian metrics on a compact manifold.
\end{abstract}  
\maketitle
\vspace*{-7mm}
In a series of papers (~\cite{k,k2,ka}) the author has studied a new class of generalized Fr\'{e}chet manifolds. In this short note we give an almost evident application
of these manifolds to study geometric evaluations equations.

A Ricci flow  is a family  of Riemannian metrics $g(t)$ on a smooth manifold, 
parametrized by a time interval $I \subset \R$ and evolving by 
$$
\dfrac{\partial g(t)}{\partial t} = -2 \ri (g(t)),
$$
where $\ri (g(t))$ is the Ricci curvature of the metric $g(t)$. One considers $t$ as time and studies the equation as an initial
value problem, that is to say find a one-parameter family $(M,g(t))$ of Riemannian manifolds with $g(0)=g_0$ satisfying
the Ricci flow equation. The Ricci flow is a weakly parabolic system, therefore, the short-time existence  does not come from the 
general PDE theory. The original  proof of the local existence and uniqueness for solutions to the  Ricci flow on a compact manifold  by Hamilton (\cite{h}) is sophisticated and involves the tame Fr\'{e}chet setting and 
the Nash-Moser inverse function theorem. Later, DeTurck~\cite{d} simplified the proof by adding to the 
equation a separate evolving diffeomorphism to make it a strongly parabolic and deduced the local existence and uniqueness from an equivalent equation which is strictly parabolic, and the classical inverse function theorem suffices.

The key ingredient to study geometric evaluation equations is a structure of the space of all Riemannian metrics on a given manifold which is
a Fr\'{e}chet manifold. There are some intrinsic difficulties with Fr\'{e}chet spaces such as the lack of a general theory for ODEs and
an inverse function theorem (in general) that make the proofs of results (like the Hamilton's proof) involving these spaces formidable. However, 
there are some ways out of these difficulties. One way is a projective limit approach where we consider a special type of 
Fr\'{e}chet manifolds, including the space of Riemannian metrics, which are obtained as projective limit of Banach manifolds (see~\cite{ctc}). Using this method a simpler proof for the short-time existence for Ricci flow was established in~\cite{ir}.  But this approach
also has a difficulty, one needs to establish the existence of  projective limits of Banach corresponding factors of geometrical and topological objects which would not be  always easy.

In this article we give a simple proof for  the local existence for Ricci flow. We give to the space of Riemannian metrics
on a compact manifold the structure of a generalized Fr\'{e}chet manifold, bounded (or $MC^k$)  Fr\'{e}chet manifold, introduced in~\cite{muller}. These generalized manifolds surpass the  calculus and  geometry of  Fr\'{e}chet manifolds and we
would expect more of their applications to problems in global analysis  (cf.~\cite{k2}). 
\section{RESULT}
For the reader's convenience, we briefly recall the bounded Fr\'{e}chet setting. For more details we refer to~\cite{k2}.

Let $(F,d)$ be a Fr\'{e}chet space whose topology is defined by a complete translation-invariant metric $d$. 
A metric with absolutely convex balls will be called a standard metric. 
Every Fr\'{e}chet space  admits a standard metric which defines its topology. We shall always define the
topology of Fr\'{e}chet spaces with this type of metrics.

Let $(E,g)$ and $(F,d)$ be Fr\'{e}chet spaces and let $\mathcal{L}_{g,d}(E,F)$ be the set of all 
linear maps $ L: E \rightarrow F $ such that 
\begin{equation*}
\mathpzc{Lip} (L )_{g,d}\, \coloneq \displaystyle \sup_{x \in E\setminus\{0\}} \dfrac{d (L(x),0)}{g( x,0)} < \infty.
\end{equation*}
The translation-invariant metric 
\begin{equation} \label{metric}  
 D_{g,d}: \mathcal{L}_{g,d}(E,F) \times \mathcal{L}_{g,d}(E,F) \longrightarrow [0,\infty) , \,\,
(L,H) \mapsto \mathpzc{Lip}(L -H)_{g,d} \,,
\end{equation}
on $\mathcal{L}_{d,g}(E,F)$ turns it  into an Abelian topological group (\cite[Remark 2.1]{k}).
Let $ E,F $ be Fr\'{e}chet spaces, $ U $ an open subset of $ E $, and $ P:U \rightarrow F $
 a continuous map. Let $CL(E,F)$ be the space of all continuous linear maps from $E$ to $F$ topologized by the compact-open topology. 
 We say $ P $ is differentiable at the point $ p \in U$ if the directional derivative  
$\operatorname{d}P(p)$ exists in all directions $ h \in E $. 
If $P$ is differentiable at all points $p \in U$, if $\operatorname{d}P(p) : U \rightarrow CL(E,F)$ is continuous for all
$p \in U$ and if the induced map $ P': U \times E \rightarrow F,\,(u,h) \mapsto \operatorname{d}P(u)h $
 is continuous in the product topology, then we say that $ P $ is Keller-differentiable.
 We define $ P^{(k+1)}: U \times E^{k+1} \rightarrow F $ in the obvious inductive fashion.

If $P$ is Keller-differentiable, $ \operatorname{d}P(p) \in \mathcal{L}_{d,g}(E,F) $ for all $ p \in U $, and the induced map 
$ \operatorname{d}P(p) : U \rightarrow \mathcal{L}_{d,g}(E,F)   $ is continuous, then $ P $ is called bounded differentiable. We say $ P $ is $ MC^{0} $ 
and write $ P^0 = P $ if it is continuous. 
We say $P$ is an $ MC^{1} $ and write  $P^{(1)} = P' $ if it is bounded differentiable. Let $ \mathcal{L}_{d,g}(E,F)_0 $ be 
the connected component of $ \mathcal{L}_{d,g}(E,F) $ containing the zero map. If $ P $ is bounded differentiable and if 
 $V \subseteq U$ is a connected open neighborhood of $x_0 \in U$, then $P'(V)$ is connected and hence contained in the connected component
$P'(x_0) +  \mathcal{L}_{d,g}(E,F)_0 $ of $P'(x_0)$ in $\mathcal{L}_{d,g}(E,F)$. Thus, $P'\mid_V - P'(x_0):V \rightarrow \mathcal{L}_{d,g}(E,F)_0 $
 is again a map between subsets of Fr\'{e}chet spaces. This enables a recursive definition: If $P$ is $MC^1$ and $V$ can be chosen for each
 $x_0 \in U$ such that $P'\mid_V - P'(x_0):V \rightarrow \mathcal{L}_{d,g}(E,F)_0 $ is $ MC^{k-1} $, 
then $ P $ is called an $ MC^k$-map. We make a piecewise definition of $P^{(k)}$ by $ P^{(k)}\mid_V \coloneq \left(P'\mid_V - P'(x_0)\right)^{(k-1)} $
for $x_0$ and $V$ as before. 
The map $ P $ is $ MC^{\infty} $ (or smooth) if it is $ MC^k $ for all $ k \in \mathbb{N}_0 $. We shall denote the derivative of $P$ at $p$ by $\operatorname{D}P(p)$.

Within this framework we can define $\mc$-Fr\'{e}chet manifolds, $\mc$-maps of manifolds
and tangent bundle over $\mc$-Fr\'{e}chet manifolds in obvious fashion way.

Given a closed connected smooth Riemannian manifold $(M,g)$, we want to endow the set $\me$ of all Riemannian metrics on it
with the structure of a bounded Fr\'{e}chet manifold. To get this manifold structure we need conditions of bounded geometry 
(cf. \cite{cg,sh}).
\begin{defn}
We will say that $(M,g)$ has bounded geometry, if
\begin{description}
 \item[i] the injectivity radius is positive, and 
 \item[ii] the curvature tensor of $M$ and all its covariant derivatives are bounded. 
\end{description}
A vector (or fiber) bundle $\pi: E\rightarrow M$ equipped with a Riemannian bundle metric and a compatible connection
will be said to have a bounded geometry, if
\begin{description}
\item[iii] the curvature tensor of $E$ and all its covariant derivatives are bounded.
\end{description}
\end{defn}
Every bundle over a compact manifold has a bounded geometry. Furthermore, Trivial bundle, tangent and cotangent bundles, and
tensor bundles are example of bundles of bounded geometry (see~\cite{sh}). Any manifold admits a metric of bounded geometry
(\cite{gr}). Analogously, any vector bundle can be equipped with a Riemannian bundle metric and a compatible connection
such that it becomes of bounded geometry (\cite[Lemma 2.9]{eg}).

The bounded Fr\'{e}chet manifold structure of $\me$ is a direct consequence of the following theorem.
\begin{theorem}\cite[Theorem 3.35]{muller}\label{boun}
Let $\pi :E \rightarrow M$ be a fiber bundle over $M$ equipped
with a Riemannian fiber bundle metric and a fibre bundle such that the bundle is of
bounded geometry. Then the space of smooth sections $\Gamma(\pi)$ can be given the structure of a bounded
Fr\'{e}cht manifold.
\end{theorem}
\begin{cor}
The space $\me$ has the structure of a bounded Fr\'{e}chet manifold.
Furthermore, $ \me $ 
is an $ MC^{\infty} $-Fr\'{e}chet manifold.
\end{cor}
\begin{proof}
Consider the bundle $\pi : T^*M \otimes  T^*M \rightarrow M$ (symmetric tensor product of the cotangent bundle $T^*M$).
We can endow this bundle with a Riemannian bundle metric $<.,.>$  of bounded geometry. The space
$\me$ is the space of all sections of the fiber bundle 
$\mathrm{Riem}(M)$ defined pointwise as the positive definite elements (pointwise an open convex cone). Now,
any section of $\pi$ with values in $\mathrm{Riem}(M)$ has an open neighborhood $\mathcal{U}$ contained in 
$\mathrm{Riem}(M)$, and as $\pi$ is a vector bundle with a metric of bounded geometry, we can conclude that $\mathrm{Riem}(M)$
is of bounded geometry. Therefore, Theorem~\ref{boun} follows that $\me$ the structure of a bounded Fr\'{e}chet manifold. Since transition mappings are globally Lipschitz, it follows by  \cite[Lemma B.1]{gl} that they are of class $ MC^{\infty} $ and therefore $ \me$ is an $ MC^{\infty} $-Fr\'{e}chet manifold. 
\end{proof}
A vector field on an infinite dimensional Fr\'{e}chet manifold may have no, one ore multiple  integral curves locally.
However, a vector field on a bounded Fr\'{e}chet manifold always has a unique integral curve. More precisely, 
\begin{theorem}\cite[Theorem 5.1]{ka}\label{ex}
Let $\xi :M \rightarrow TM$ be an $ MC^k $-vector field, $ k\geq1 $. Then  there exits an integral curve for $\xi$ at $p \in M$. Furthermore, any two such curves are equal on the intersection of their domains.
\end{theorem}
As an immediate corollary we get the following result. 
\begin{theorem}
 	Let $ M $ be a closed Einstein manifold. 
There exists a unique solution $g(t)$ to the Ricci flow equation for $t \in [0,T)$, where $T$ depends on the initial metric.
\end{theorem}
\begin{proof}
	The manifold $ \me $ of Riemannian metrics of $ M $ is an $ MC^{\infty} $-Fr\'{e}chet manifold by Corollary \ref{boun}.
Consider the vector field
\begin{equation}\label{eq:ricci}
g \mapsto -2\mathrm{Ric}(g)=-2kg,
\end{equation}
where $ g $ is a Riemannian metric and $ k $ is a constant.
Locally it is linear and globally Lipschitz. Therefore, by \cite[Lemma B.1]{gl}  is $ MC^{\infty} $.
Then the Ricci flow curve is the integral curve of this vector field. Hence  Theorem~\ref{ex} implies
its  local existence and uniqueness with any initial metric at $t=0$.
\end{proof}


\begin{thebibliography}{1} 
\bibitem{cg}
J.~Cheeger and M.~Gromov, Bounds on the von Neumann dimension of $L^2$-cohomology and the Gauss-Bonnet theorem
for open manifolds, J.~Differential Geometry 21 (1985),1-34.
\bibitem{d}
D.~M.~DeTurck, Deforming metrics in the direction of their Ricci tensors, J.
Differ. Geom. 18 (1983), 157-162.
\bibitem{ctc}
C.~T.~J.~Dodson, G.~Galanis and E.~Vassiliou,  Geometry in a Fr\'{e}chet Context: A projective limit approach, London mathematical
society lectures notes series 428, 2015.
\bibitem{ka}
K.~Eftekharinasab, Geometry of bounded Fr\'{e}chet manifolds, Rocky Mountain Journal of Mathematics 46 (2016), 895-913
\bibitem{k}
K.~Eftekharinasab, Sard's theorem for mappings between Fr\'{e}chet manifolds, Ukrainian Math. J. 62 (2010), 1634-1641.
\bibitem{k2}
K.~Eftekharinasab, Transversality and Lipschitz-Fredholm maps, Zb. Pr. Inst. Mat. NAN Ukr. Vol. 12, 1 (2015), 89-104. 
\bibitem{eg}
A.~Engel, Indices of pseudo-differential operators on open manifolds, Ph.D. thesis, University of Augsburg, 2014.
\bibitem{ir}
H. Ghahremani-Gol, A. Razavi, Ricci flow and the manifold of Riemannian metrics, Balkan
Journal of Geometry and Its Applications, Vol. 18, 2 (2013), pp. 20-30.
\bibitem{gr}
R.~Greene, Complete metrics of bounded curvature on noncompact
manifolds, Archiv der Mathematik Vol. 31, 1 (1978),  89-95.
\bibitem{h}
R.~Hamilton, Three-manifolds with positive Ricci curvature, J.~ Differ.~ Geom. 17
(1982), 255-306.
\bibitem{muller}
M\"{u}ller O. A metric approach to Fr\'{e}chet geometry, Journal of Geometry and physics 11 (2008), 1477 -1500.
\bibitem{sh}
M.~A.~Shubin, Spectral Theory of Elliptic Operators on Non-Compact
Manifolds, Ast\'{e}risque 207 (1992), 35–108.
\bibitem{gl}
H. Gl\"{o}ckner, Implicit Functions From Topological Vector Spaces to Fr\'{e}chet
Spaces in the Presence of Metric estimates,  arXiv: 2112.08114 [math.DG].
\end{thebibliography}
\end{document}